\documentclass[12pt]{amsart}
\usepackage{amsmath}
\usepackage{amssymb}

\newtheorem{theorem}{Theorem}[section]
\newtheorem{corollary}[theorem]{Corollary}
\newtheorem{lemma}[theorem]{Lemma}
\newtheorem{proposition}[theorem]{Proposition}
\theoremstyle{definition}
\newtheorem{definition}[theorem]{Definition}

\newtheorem{remark}[theorem]{Remark}

\numberwithin{equation}{section}

\usepackage{enumerate}
\def\ees{{\accent"5E e}\kern-.385em\raise.2ex\hbox{\char'23}\kern-.08em}
\def\EES{{\accent"5E e}\kern-.5em\raise.8ex\hbox{\char'23 }}
\def\ow{o\kern-.42em\raise.82ex\hbox{\vrule width .12em height .0ex depth .075ex \kern-0.16em \char'56}\kern-.07em}
\def\OW{o\kern-.460em\raise1.36ex\hbox{
\vrule width .13em height .0ex depth .075ex \kern-0.16em
\char'56}\kern-.07em}
\def\DD{D\kern-.7em\raise0.4ex\hbox{\char '55}\kern.33em}
\def \cm {\mathrm{cone} M}

\def \cl {\mathrm{cone} L}
\def \X {\mathbb R^n}
\def \Y {\mathbb R^m}
\def \SX {\mathbb S^{n-1}}
\def \SY {\mathbb S^{m-1}}
\def \BX {\mathbb B^{n}}
\def \BY {\mathbb B^{m}}
\def \BBX {\overline{\mathbb B}^{n}}
\def \BBY {\overline{\mathbb B}^{m}}

\title[Characterizations of directional openness]{Characterizations of directional openness for set-valued mappings}

\author[S. T. \DD inh]{S\~I TI\d{\^E}P \DD INH}
\address[S. T. \DD inh]{Institute of Mathematics, VAST, 18, Hoang Quoc Viet Road, Cau Giay District 10307, Hanoi, Vietnam}
\email{\tt dstiep@math.ac.vn}

\author[T. S. Ph\d{a}m]{TI\EES N-S\OW N PH\d{A}M}
\address[T. S. Ph\d{a}m]{Department of Mathematics, Dalat University, 1 Phu Dong Thien Vuong, Dalat, Vietnam}
\email{\tt sonpt@dlu.edu.vn}

\dedicatory{Dedicated to Professor Hubertus Jongen  on the occasion of  his 75th birthday}


\date{\today}

\begin{document}

\begin{abstract}
We provide necessary and sufficient conditions for a set-valued mapping between finite dimensional spaces to be directionally open by relating this property with directional regularity, H\"older continuity of the inverse mapping, coderivatives and variations. 
These generalize and refine some previously known results.
\end{abstract}

\maketitle


\section{Introduction}

This paper concerns the following three well posedness properties of set-valued mappings between finite dimensional spaces: {\em openness, metric regularity and Lipschitz/H\"older continuity.} These properties are fundamental in many areas of variational analysis and its applications, and received a huge amount of attention over the years. For more details, we refer the reader to the comprehensive monographs 
\cite{Aubin1990, Bonnas2000, Dontchev2009, Ioffe2017, Klatte2002, Mordukhovich2006, Mordukhovich2018, Rockafellar1998, Schirotzek2007} with the references therein.

 Let $F \colon \X \rightrightarrows \Y$ be a closed set-valued mapping, $\bar{x} \in \X$ and $\bar{y} \in F(\bar{x}).$
From the work of Penot~\cite{Penot1989}, the following properties are equivalent (terminology will be explained later):
\begin{enumerate}[{\rm (i)}]
\item the mapping $F$ is linearly open around $(\bar{x}, \bar{y});$
\item the mapping $F$ is metrically regular around $(\bar{x}, \bar{y});$
\item the inverse mapping $F^{-1}$ is pseudo-Lipschitz continuous (known also as Aubin continuous) around $(\bar{y}, \bar{x}).$
\end{enumerate}
Furthermore, these properties may be characterized via the {\em coderivative} of $F$ at $(\bar{x}, \bar{y})$ as shown by Mordukhovich~\cite{Mordukhovich1993}; see also \cite[Chapter~3]{Mordukhovich2018}.

Thanks to Borwein and Zhuang~\cite{Borwein1988} (see also \cite{Ioffe2013, Yen2008}), it is well known that for any strictly increasing continuous function $\phi \colon [0, +\infty) \to [0, +\infty)$ vanishing at $0,$ the following properties are equivalent:
\begin{enumerate}[{\rm (i)}]
\item the mapping $F$ is $\phi$-open around $(\bar{x}, \bar{y});$
\item the mapping $F$ is $\phi^{-1}$-regular around $(\bar{x}, \bar{y});$
\item the inverse mapping $F^{-1}$ is $\phi^{-1}$-continuous around $(\bar{y}, \bar{x}).$
\end{enumerate}
Moreover, when the function $\phi$ has the form $c t^r$ for some $c > 0$ and $r \geqslant 1,$ these properties may be characterized in terms of the $r$-order {\em variation} of $F$ at $(\bar{x}, \bar{y});$ see also \cite{Frankowska1987}.

In recent years, significant progress has been made by several authors to go beyond the well posedness properties mentioned above, which imply the action of mappings around the reference points in all directions, to the case where the relations which define these properties hold only on some directions
(see \cite{Arutyunov2006, Arutyunov2007, Gfrerer2013, Gfrerer2016-2, Huynh2015, Ioffe2017, Penot1999} and the references therein). 
In particular, Frankowska and Quincampoix~\cite[Theorem~5.2]{Frankowska2012} presented a necessary condition and a sufficient condition for the H\"older metric regularity with respect to a set of directions belonging to a closed convex cone in {\em output spaces} using high order variations.

More recently, motivated by some optimization problems such as differentiating between minima and maxima, 
Durea, Pan\c{t}iruc and Strugariu~\cite{Durea2017} introduced and studied the following three directional concepts for set-valued mappings that take into account sets of directions in {\em both input and output  spaces}: directional linear openness, directional metric regularity and directional Aubin continuity; they show their links on the lines of the classical case and give necessary conditions and sufficient conditions for these directional concepts in terms of coderivatives.

The purpose of this paper is to generalize and refine the aforementioned results.
More precisely, our main contributions are as follows:
\begin{enumerate}[{\rm (i)}]
\item Show equivalence between the following three directional concepts: directional $\phi$-openness, directional $\phi$-regularity and directional $\phi$-continuity (Theorem~\ref{Thm1}). 

\item Give a necessary and sufficient condition in terms of directional coderivatives for a closed set-valued mapping to be directionally linearly open (Theorem~\ref{Thm2}).

\item Provide a necessary and sufficient condition in terms of directional variations for a closed set-valued mapping to be directionally linear/nonlinear open (Theorem~\ref{Thm3}).
\end{enumerate}

Note that some results can be given in normed vector spaces but we prefer to work with finite dimensional spaces for simplicity of presentation,

It would be interesting to have a {\em necessary and sufficient condition} in terms of (directional) coderivatives for (directional) {\em nonlinear} openness. 
This question is mostly open, to the best of the authors' knowledge; see also \cite{Dinh2021} and the references therein.

The rest of this paper is organized as follows.
The definitions of directional openness, directional regularity and directional continuity are given in Section~\ref{SectionPreliminary}. 
The notions of directional coderivative and directional variation are also introduced in this section. 
The main results and their proofs will be presented in Section~\ref{Main}.

\section{Preliminaries} \label{SectionPreliminary}

Let $\mathbb{R}^n$ be equipped with the usual scalar product $\langle \cdot, \cdot \rangle$ and the corresponding Euclidean norm $\| \cdot\|.$ 
The open ball and sphere centered at $x\in\mathbb{R}^n$ of radius $r$ will be denoted by $\mathbb{B}^n(x,r)$ and $\mathbb{S}^{n - 1}(x,r)$  respectively.  For simplicity, we write $\mathbb B^n$ and $\mathbb S^{n-1}$ if $x=0$ and $r=1$.

For a subset $\Omega$ of $\mathbb{R}^n,$ the closure of $\Omega$ will be written as $\overline{\Omega},$ the cone at the origin generated by $\Omega$ is designated by cone$\Omega.$ To simplify notation, for a point $x \in \mathbb{R}^n,$ we write $[x + \Omega]$ instead of the set $\{x + u \ : \ u \in \Omega\}.$

Let $F\colon \X \rightrightarrows \Y$ be a set-valued mapping.
The graph of $F$ is denoted by 
$$\mathrm{graph} F: = \{(x, y) \in \X \times \Y \ :\ y \in F (x)\}.$$ 
The mapping $F$ is called {\em closed} if its graph is a closed set.
The inverse set-valued mapping of $F$ is $F^{-1} \colon \Y \to \X$ given by 
$$F^{-1}(y) := \{x \in \X \ : \ y \in F(x)\}.$$
We shall consider the set
$$\limsup_{x \overset{\Omega} {\to} \bar{x}} F(x) \ :=\ \left\{ 
\begin{array}{lll}
y \in \Y  &:   & \textrm{there are sequences } x_k \to \bar{x}, \ y_k \to y \\
&& \textrm{with } x_k \in \Omega \textrm{ and } y_k \in F(x_k)
\end{array}\right\},$$
which is called the {\em Painlev\'e--Kuratowski upper limit} of $F$ at $\bar{x}$ (along $\Omega$).

\subsection{Directional minimal time function}

We will make use of a special minimal time function with respect to a set of directions, which was defined and analyzed in \cite{Durea2016}.

\begin{definition}
Let $L$ and $\Omega$ be nonempty subsets of $\SX$ and $\X$ respectively. 
Then the function
\begin{equation*}
\begin{array}{lrl}
T_L (x, \Omega)& := &\inf \{t \geqslant 0\ : \ \text{ there is } u\in L \text{ such that } x + tu \in \Omega\}\\
&= &\inf \{t \geqslant 0\ :\ [x + t L] \cap \Omega \ne \emptyset\}
\end{array}
\end{equation*}
is called the {\em directional minimal time function} with respect to $L$.
\end{definition}
By convention, set $T_L (x,\emptyset) := +\infty$ for every $x\in\X.$ Moreover, for simplicity, we denote  $T_L (x, \{x'\})$ by $T_L (x, x').$

\begin{remark}
By definition, if $L = \SX,$ then $T_L (\cdot, \Omega)$ is the distance function to the set $\Omega.$ Furthermore, if the sets $L$ and $\Omega$ are closed, then the infimum in the definition of $T_L (\cdot, \Omega)$ is always attained.
\end{remark}

\subsection{Directional openness, regularity and continuity}

In the current and subsequent sections, let $\Phi$ denote the set of all strictly increasing continuous functions $\phi \colon [0, +\infty) \to [0, +\infty)$ with $\phi(0) = 0.$

\begin{definition}\label{def}
Let $F \colon \mathbb{R}^n \rightrightarrows \mathbb{R}^m$ be a set-valued mapping, 
$(\bar{x}, \bar{y}) \in \mathrm{graph} F,$ $L$ and $M$ be nonempty subsets of $\SX$ and $\SY,$ respectively, and let $\phi \in \Phi.$ 

\begin{enumerate}[{\rm (i)}]
\item $F$ is {\em directionally $\phi$-open around $(\bar{x}, \bar{y})$ with respect to $L$ and $M$} if there are $\epsilon > 0$ and open neighborhoods $U$ of $\bar{x}$ and $V$ of $\bar{y}$ such that for every $t \in (0, \epsilon)$ and every $(x, y) \in (U \times V) \cap \mathrm{graph}F,$
$$\BY(y, \phi(t)) \cap [y + \cm] \subset F \big (\BX(x, t) \cap [x + \cl] \big).$$

\item $F$ is {\em directionally $\phi$-regular around $(\bar{x}, \bar{y})$ with respect to $L$ and $M$} if there are $\epsilon > 0$ and open neighborhoods $U$ of $\bar{x}$ and $V$ of $\bar{y}$ such that for every $(x, y) \in U \times V,$
with $T_M(y, F(x)) < \epsilon,$
\begin{eqnarray*}
T_L(x, F^{-1}(y)) & \leqslant & \phi \big( T_M(y, F(x)) \big).
\end{eqnarray*}

\item $F$ is {\em directionally $\phi$-continuous around $(\bar{x}, \bar{y})$ with respect to $L$ and $M$} if there are open neighborhoods $U$ of $\bar{x}$ and $V$ of $\bar{y}$ such that for every $x, x' \in U,$ and every $y \in F(x)\cap V$, one has
\begin{eqnarray*}
T_M\big(y, F(x') \big) & \leqslant & \phi \big( T_L(x', x) \big).
\end{eqnarray*}
\end{enumerate}
\end{definition}

\begin{remark}{\rm 
Unlike~\cite{Durea2017}, in Definition~\ref{def}(i) above, we write $\cm$ instead of $-\cm$ since it seems more natural this way; see also the formulation of Theorem~\ref{Thm1}.
}\end{remark}

The most interesting candidate for our strictly increasing continuous function $\phi(t)$ is $c t^r$ for some $c  > 0$ and  $r \geqslant 1.$ 

\begin{definition}
Let $F \colon \mathbb{R}^n \rightrightarrows \mathbb{R}^m$ be a set-valued mapping, 
$(\bar{x}, \bar{y}) \in \mathrm{graph} F,$ $L$ and $M$ be nonempty subsets of $\SX$ and $\SY,$ respectively, and let $c > 0$ and $r \geqslant 1.$

\begin{enumerate}[{\rm (i)}] 
\item $F$ is {\em directionally open} (resp.,  {\em directionally regular} and {\em directionally continuous})
{\em at rate $r$ with modulus $c$ around $(\bar{x}, \bar{y})$ with respect to $L$ and $M$} if $F$ is directionally $\phi$-open (resp., directionally $\phi$-regular and directionally $\phi$-continuous) around $(\bar{x}, \bar{y})$ with respect to $L$ and $M$ for $\phi = c t^r.$

\item $F$ is {\em directionally linearly open} (resp.,  {\em directionally metrically regular} and {\em directionally Aubin continuous})
{\em  with modulus $c>0$ around $(\bar{x}, \bar{y})$ with respect to $L$ and $M$} if $F$ is directionally $\phi$-open (resp., directionally $\phi$-regular and directionally $\phi$-continuous)
 around $(\bar{x}, \bar{y})$ with respect to $L$ and $M$ for $\phi = c t.$
\end{enumerate}
\end{definition}

\begin{remark}{\rm 
Assume that $F \colon \mathbb{R}^n \rightrightarrows \mathbb{R}^m$ is semi-algebraic set-valued mapping and $L, M$ are semi-algebraic sets\footnote{A subset of $\mathbb{R}^n$ is {\em semialgebraic} if it can be written as a finite union and
intersection of sets of the form $\{x \in \mathbb{R}^n \ : \ f(x) = 0\}$ and 
$\{x \in \mathbb{R}^n \ : \ f(x) > 0\},$ where $f$ is a polynomial function on $\mathbb{R}^n;$ 
a set-valued mapping $F \colon \mathbb{R}^n \rightrightarrows \mathbb{R}^m$ is {\em semialgebraic} if its graph is a semialgebraic subset of
$\mathbb{R}^n \times \mathbb{R}^m.$ (See, e.g., \cite{HaHV2017}.)}.
Then it is not hard to check that $F$ is directionally $\phi$-open around $(\bar{x}, \bar{y})$ with respect to $L$ and $M$ if  and only if there are $\epsilon > 0$ and open neighborhoods $U$ of $\bar{x}$ and $V$ of $\bar{y}$ such that for every $t \in (0, \epsilon)$ there exists $s > 0$ satisfying 
$$\BY(y, s) \cap [y + \cm] \subset F \big (\BX(x, t) \cap [x + \cl] \big)$$
for every $(x, y) \in (U \times V) \cap \mathrm{graph}F.$ Furthermore, the function $\phi$ can be chosen to be $ct^r$ for some $c > 0$ and $r \geqslant 1.$ As we shall not use these facts, we leave the proof to the reader; see also \cite{Lee2022}.
}\end{remark}

\subsection{Directional normal cones and coderivatives}

In order to characterize the directional linear openness in terms of coderivatives, we recall the following notions; confer \cite{Mordukhovich2018}.

\begin{definition}
Let $\Omega \subset \mathbb R^n\times\mathbb R^m,$ $(x, y) \in \Omega,$ $L$ and $M$ be nonempty subsets of $\SX$ and $\SY,$ respectively.
\begin{enumerate}[{\rm (i)}]
\item The {\em regular (or Fr\'echet)  normal cone to $\Omega$ at ${(x,y)}$ with respect to $L$ and $M,$} denoted by $\widehat{N}_{L,M} (\Omega, {(x,y)}),$ is the set of vectors $(x^*, y^*) \in \mathbb R^n\times\mathbb R^m$ such that for every $\epsilon > 0$ there exists $\delta > 0$ such that
\begin{eqnarray*}
\langle x^*, x' - {x} \rangle + \langle y^*, y' - {y} \rangle & \leqslant & \epsilon \big (\|x' - {x}\| + \|y' - {y}\| \big)
\end{eqnarray*}
whenever $(x',y') \in \Omega$ with $x' \in \BX({x}, \delta) \cap [{x} + \cl]$ and $y' \in \BY({y}, \delta) \cap [{y} + \cm].$

\item The {\em limiting (or Mordukhovich) normal cone to $\Omega$ at $(\bar{x},\bar{y})$ with respect to $L$ and $M$} is given by
\begin{eqnarray*}
N_{L, M}(\Omega, (\bar{x},\bar{y})) &:=& \limsup_{(x, y) \overset{\Omega}{\to} (\bar{x},\bar{y})} \widehat{N}_{L,M} (\Omega, {(x,y)}),
\end{eqnarray*}
where $\limsup$ is the Painlev\'e--Kuratowski upper limit.
\end{enumerate}
\end{definition}

\begin{definition}
Let $F \colon \X \rightrightarrows \Y$ be a set-valued mapping, $(\bar{x}, \bar{y}) \in \mathrm{graph} F,$ $L$ and $M$ be nonempty subsets of $\SX$ and $\SY,$ respectively.
\begin{enumerate}[{\rm (i)}]
\item The {\em regular (or Fr\'echet) directional coderivative of $F$ at $(\bar{x}, \bar{y})$ with respect to $L$ and $M$} is the set-valued mapping
$$\widehat{D}^*_{L, M} F(\bar{x}, \bar{y}) \colon \Y \rightrightarrows \X$$ 
defined by
$$\widehat{D}^*_{L, M} F(\bar{x}, \bar{y})(y^*) := \{
x^* \in \X  :  (x^*, -y^*)  \in \widehat{N}_{L, M} (\mathrm{graph} F, (\bar{x}, \bar{y}))\},$$
where $y^* \in \Y.$

\item The {\em limiting (or Mordukhovich) directional coderivative of $F$ at $(\bar{x}, \bar{y})$ with respect to $L$ and $M$} is the set-valued mapping
$${D}^*_{L, M} F(\bar{x}, \bar{y}) \colon \Y \rightrightarrows \X$$ 
given by
$${D}^*_{L, M} F(\bar{x}, \bar{y})(y^*) := \{x^* \in \X  :  (x^*, -y^*)  \in {N}_{L, M} (\mathrm{graph} F, (\bar{x}, \bar{y}))\},$$
where $y^* \in \Y.$
\end{enumerate}
\end{definition}

If $L = \SX$ and $M = \SY,$ for simplicity, we will omit the subscripts $L$ and $M$ in the above definitions.

\subsection{Subdifferentials}

\begin{definition}
Let $f\colon\X \to \mathbb R \cup \{+\infty\}$ be a lower semicontinuous function which is finite at $\bar x \in \X.$
The {\em limiting (or Mordukhovich) subdifferential of $f$ at $\bar x$} is defined by
$$\partial f(\bar x) := \{x^* \in \mathbb R^n\ :\ (x^*, -1)\in N(\mathrm{epi} f,(\bar x,f(\bar x)))\},$$
where $\mathrm{epi} f:=\{(x,y)\in\mathbb R^n\times\mathbb R\ : \ y\geqslant f(x)\}$ is the {\em epigraph} of $f.$
\end{definition}

Let us recall some properties of subdifferentials which will be needed later (see~\cite[Proposition~1.30, Corollary~2.20]{Mordukhovich2018}).

\begin{proposition}\label{subdifferential}
Let $f \colon\X \to \mathbb R \cup \{+\infty\}$ be a lower semi-continuous function which is finite at $\bar x \in \X.$ The following assertions hold:
\begin{enumerate}[{\rm (i)}]
\item If $\bar x$ is a local minimizer of $f$, then $0 \in \partial f(\bar x)$.
\item If $g \colon\X \to \mathbb R \cup \{+\infty\}$ is locally Lipschitz around $\bar x$, then 
$$\partial (f + g)(\bar{x}) \subset \partial f(\bar{x}) + \partial g(\bar{x}).$$
\item We have the representation
$$\partial(\|\cdot - \bar x\|)(x) =
\begin{cases}
\displaystyle\frac{x - \bar x}{\|x-\bar x\|} & \textrm{ if } x \ne \bar x,\\
\overline{\mathbb{B}}^n & \textrm{ otherwise.}
\end{cases}$$
\end{enumerate}
\end{proposition}

\subsection{Directional variations}

In order to characterize the directional linear/nonlinear openness in terms of variations, we need the following concept.

\begin{definition}
Let $F \colon \X \rightrightarrows \Y$ be a set-valued mapping, $(\bar{x}, \bar{y}) \in \mathrm{graph} F,$ $L$ and $M$ be nonempty subsets of $\SX$ and $\SY,$ respectively, and let $r > 0.$ The {\em $r$-th directional variation of $F$ at $(\bar{x}, \bar{y})$ with respect to $L$ and $M,$} denoted by $F^{(r)}_{L, M}(\bar{x}, \bar{y}),$ is the set of vectors $v \in \cm$ such that for all sequences $t_k > 0$ and $(x_k, y_k) \in \mathrm{graph}F$ converging to zero and $(\bar{x}, \bar{y})$ respectively, there exists a sequence $v_k \in \mathbb{R}^m$ tending to $v$ such that
\begin{eqnarray*}
&& v_k \in \frac{F \big (\mathbb{B}(x_k, t_k) \cap [x_k + \cl]\big) - y_k}{t_k^r} \quad \textrm{ and } \quad v \in v_k + \cm.
\end{eqnarray*}
\end{definition}

\begin{remark}
Observe that this definition is nothing else than a directional version of  high order variations, which was introduced by Frankowska~\cite{Frankowska1987}.
\end{remark}

\subsection{A directional Ekeland variational principle}

We recall here a directional Ekeland variational principle, which makes use of the directional minimal time function and which will be an important tool in the proof of Theorem~\ref{Thm2}. 

\begin{proposition}[{see \cite[Corollary~3.4]{Durea2017}}] \label{Ekeland}
Let $L \subset \SX$ and $M \subset \SY$ be nonempty closed sets such that $\cl$ and $\cm$ are convex. 
Let $\Omega\subset \X \times \Y$ be a closed set and $f \colon \Omega \to\mathbb R \cup \{+\infty\}$ be a bounded from below and lower semi-continuous function, which is finite at $(x_0, y_0) \in \Omega.$ 
Then for every $\epsilon > 0$, there exists $(x_{\epsilon}, y_{\epsilon}) \in \Omega$ such that
$$f(x_{\epsilon}, y_{\epsilon}) \leqslant f(x_0, y_0) - {\epsilon} \big(T_L(x_{\epsilon}, x_0) + T_M (y_{\epsilon}, y_0) \big)$$
and for any $(x, y) \in \Omega \setminus \{(x_{\epsilon}, y_{\epsilon})\},$
$$f(x_{\epsilon}, y_{\epsilon}) < f(x,y) + {\epsilon} \big(T_L(x, x_{\epsilon}) + T_M (y, y_{\epsilon}) \big).$$
\end{proposition}

\section{Main results}\label{Main}

\subsection{Directional openness, regularity and continuity}

As in the classical case, the relations between the concepts given before are presented in the following result.

\begin{theorem} \label{Thm1}
Let $F \colon \mathbb{R}^n \rightrightarrows \mathbb{R}^m$ be a set-valued mapping, 
$(\bar{x}, \bar{y}) \in \mathrm{graph} F,$ $L$ and $M$ be nonempty subsets of $\SX$ and $\SY,$ respectively, and let $\phi \in \Phi.$ 
Then the following statements are equivalent:
\begin{enumerate}[{\rm (i)}]
\item $F$ is directionally $\phi$-open around $(\bar{x}, \bar{y})$ with respect to $L$ and $M.$
\item $F$ is directionally $\phi^{-1}$-regular around $(\bar{x}, \bar{y})$ with respect to $L$ and $-M.$
\item $F^{-1}$ is directionally $\phi^{-1}$-continuous around $(\bar{y}, \bar{x})$ with respect to $-M$ and $L.$
\end{enumerate}
\end{theorem}

\begin{proof}
Since $\phi \in \Phi,$ the inverse function $\phi^{-1}$ exists and belongs to $\Phi,$ and so the statements make sense.

(i) $\Rightarrow$ (ii).
By definition, there is $\rho > 0$ such that for all $t \in (0, \rho)$ and all $(x, y) \in \mathrm{graph}(F) \cap \big(\BX(\bar{x}, \rho) \times
\BY(\bar{y}, \rho)\big)$ we have
\begin{eqnarray}\label{PT1}
\BY(y, \phi(t)) \cap [y + \cm] & \subset & F \big(\BX(x, t) \cap [x + \cl] \big). 
\end{eqnarray}

Let $\epsilon \in (0, \min\{\frac{\rho}{2}, \phi(\rho) \}).$ Take any $(x, y) \in \BX(\bar{x}, \epsilon) \times \BY(\bar{y}, \epsilon)$ such that $T_{-M}(y, F(x)) < \epsilon.$ We will show that
\begin{eqnarray*}
T_L(x, F^{-1}(y)) & \leqslant & \phi^{-1} \left (T_{-M}(y, F(x))  \right).
\end{eqnarray*}
To this end, take arbitrary $\delta > 0$ such that $T_{-M}(y, F(x)) + \delta < \epsilon.$
By the definition of the infimum, there exist a real number $s \in [T_{-M}(y, F(x)), T_{-M}(y, F(x))  + \delta)$ and a vector $v \in -M$ such that $y ' := y + sv \in F(x).$ We have
\begin{eqnarray*}
\|y' - \bar{y}\| 	\ =\ \|y + sv - \bar{y}\| \ \leqslant \ \|y - \bar{y}\| + s & < &  \epsilon + T_{-M}(y, F(x))  + \delta\\
&  < & 2 \epsilon \ < \ \rho.
\end{eqnarray*}
Observe that $s \in [0, \epsilon) \subset [0, \phi(\rho))$ and so $t := \phi^{-1}(s) \in [0, \rho).$ Take arbitrary $t' \in (t, \rho).$ It follows from \eqref{PT1} that
\begin{eqnarray*}
\BY(y', \phi(t')) \cap [y' + \cm] & \subset & F(\BX(x, t') \cap [x + \cl]). 
\end{eqnarray*}
Consequently, $y = y' - sv \in F(x')$ for some $x' \in \BX({x}, t') \cap [x + \cl].$ Hence
\begin{eqnarray*}
T_L(x, F^{-1}(y)) & \leqslant &  \|x - x'\| \ < \ t'.
\end{eqnarray*}
Letting $t' \to t$ we obtain
\begin{eqnarray*}
T_L(x, F^{-1}(y)) & \leqslant &  t \ = \ \phi^{-1}(s) \ \leqslant \ \phi^{-1} \left ( T_{-M}(y, F(x))  + \delta \right).
\end{eqnarray*}
Since $\delta > 0$ is chosen arbitrarily small and $\phi^{-1}$ is continuous, we have
\begin{eqnarray*}
T_L(x, F^{-1}(y)) & \leqslant & \phi^{-1} \left (T_{-M}(y, F(x))  \right).
\end{eqnarray*}

(ii) $\Rightarrow$ (iii).
By assumption, there exists $\rho > 0$ such that for every $(x, y') \in \BX(\bar{x}, \rho) \times \BY(\bar{y}, \rho)$ with
$T_{-M}(y', F(x)) < \rho$ we have
\begin{eqnarray}\label{PT2}
T_L(x, F^{-1}(y')) & \leqslant & \phi^{-1} \big( T_{-M}(y', F(x)) \big).
\end{eqnarray}

Let $\epsilon \in \big(0, \frac{\rho}{2} \big).$ Take any $y, y' \in \BY(\bar{y}, \epsilon).$ We will show that
\begin{eqnarray*}
T_L(x, F^{-1}(y')) & \leqslant & \phi^{-1} \big(  T_{-M}(y', y) \big) \quad \textrm{ for all } \quad x \in F^{-1}(y) \cap \BX(\bar{x}, \epsilon).
\end{eqnarray*}
Clearly, if $y \not \in y' - \mathrm{cone}M$ or $F^{-1}(y) \cap \BX(\bar{x}, \epsilon) = \emptyset$ then the above inequality holds trivially. So assume that $y \in y' - \mathrm{cone}M$ and $F^{-1}(y) \cap \BX(\bar{x}, \epsilon) \ne \emptyset.$ Take arbitrary 
$x \in F^{-1}(y) \cap \BX(\bar{x}, \epsilon).$ Then $x \in \BX(\bar{x}, \rho),$ $y \in F(x)$ and
\begin{eqnarray*}
T_{-M}(y', F(x)) & \leqslant & \|y' - y\| \ \leqslant \ \|y' - \bar{y}\|  + \|y - \bar{y}\| \ < \ 2 \epsilon \ < \ \rho.
\end{eqnarray*}
It follows from \eqref{PT2} that
\begin{eqnarray*}
T_L(x, F^{-1}(y')) \  \leqslant \ \phi^{-1} \big (  T_{-M}(y', F(x)) \big ) & \leqslant &  \phi^{-1}( \|y' - y\|) 
\\ & = &  \phi^{-1} \big (   T_{-M}(y', y)\big),
\end{eqnarray*}
which is the desired result.
 
(iii) $\Rightarrow$ (i).
By assumption, there exists $\rho > 0$ such that for all $y, y' \in \BY(\bar{y}, \rho)$ and all $x \in F^{-1}(y) \cap \BX(\bar{x}, \rho),$ we have
\begin{eqnarray} \label{PT3}
T_L(x, F^{-1}(y')) & \leqslant & \phi^{-1} \big(  T_{-M}(y', y) \big).
\end{eqnarray}

Let $\epsilon \in (0, \min\{\frac{\rho}{2}, \phi^{-1}(\frac{\rho}{2})\}).$ Take any $t \in (0, \epsilon)$ and any $(x, y) \in  \mathrm{graph}(F) \cap \big( \BX(\bar{x}, \epsilon) \times \BY(\bar{y}, \epsilon) \big).$ We will show that
\begin{eqnarray*}
\BY(y, \phi(t)) \cap [y + \cm] & \subset & F \big (\BX(x, t) \cap [x + \cl] \big ). 
\end{eqnarray*}
Indeed, take any $y' \in \BY(y, \phi(t)) \cap [y + \cm].$ Then $y' = y + sv$ for some $s \in [0, \phi(t))$ and some $v \in M.$ We have
\begin{eqnarray*}
\|y' - \bar{y}\|  \ = \   \|y + sv  - \bar{y}\| & \leqslant & \|y - \bar{y}\| + s  \\ & < &   \epsilon + s \ <  \ \epsilon  + \phi(t) \ < \ \epsilon + \phi(\epsilon) \ <  \ \rho.
\end{eqnarray*}
Since the function $\phi^{-1}$ is increasing, it follows from \eqref{PT3} that
\begin{eqnarray*}
T_L(x, F^{-1}(y'))  \leqslant  \phi^{-1} \big(T_{-M}(y', y) \big) = \phi^{-1}(s) < \phi^{-1}(\phi(t)) = t < +\infty.
\end{eqnarray*}
In particular, the set $F^{-1}(y')$ is nonempty and contains a point $x' \in \BX(x, t) \cap [x + \cl].$ Consequently,
$$y'  \in F(\BX(x, t) \cap [x + \cl]),$$ 
which completes the proof.
\end{proof}

The following two corollaries follow directly from Theorem~\ref{Thm1}.
\begin{corollary} \label{Corollary31}
Let $F \colon \mathbb{R}^n \rightrightarrows \mathbb{R}^m$ be a set-valued mapping, 
$(\bar{x}, \bar{y}) \in \mathrm{graph} F,$ $L$ and $M$ be nonempty subsets of $\SX$ and $\SY,$ respectively.
Then the following statements are equivalent for $c > 0$ and $r \geqslant 1$: 
\begin{enumerate}[{\rm (i)}]
\item $F$ is directionally open at rate $r$ with modulus $c$ around $(\bar{x}, \bar{y})$ with respect to $L$ and $M.$
\item $F$ is directionally regular at rate $\frac{1}{r}$ with modulus $\frac{1}{c}$ around $(\bar{x}, \bar{y})$ with respect to $L$ and $-M.$
\item $F^{-1}$ is directionally continuous at rate $\frac{1}{r}$ with modulus $\frac{1}{c}$ around $(\bar{y}, \bar{x})$ with respect to $-M$ and $L.$
\end{enumerate}
\end{corollary}

\begin{corollary}[see {\cite[Proposition~2.4]{Durea2017}}] \label{Corollary32}
Let $F \colon \mathbb{R}^n \rightrightarrows \mathbb{R}^m$ be a set-valued mapping, 
$(\bar{x}, \bar{y}) \in \mathrm{graph} F,$ $L$ and $M$ be nonempty subsets of $\SX$ and $\SY,$ respectively.
Then the following statements are equivalent for $c > 0$:
\begin{enumerate}[{\rm (i)}]
\item $F$ is directionally linearly open with modulus $c$ around $(\bar{x}, \bar{y})$ with respect to $L$ and $M.$
\item $F$ is directionally metrically regular  with modulus $\frac{1}{c}$ around $(\bar{x}, \bar{y})$ with respect to $L$ and $-M.$
\item $F^{-1}$ is directionally Aubin continuous  with modulus $\frac{1}{c}$ around $(\bar{y}, \bar{x})$ with respect to $-M$ and $L.$
\end{enumerate}
\end{corollary}

\subsection{Directional openness and coderivatives}

The following result is a directional version of the Mordukhovich criterion \cite{Mordukhovich1993}; see also \cite[Propositions~4.1, 4.2 and Theorem~4.3]{Durea2017}.

\begin{theorem} \label{Thm2}
Let $F \colon \X \rightrightarrows \Y$ be a closed set-valued mapping, 
$(\bar{x}, \bar{y}) \in \mathrm{graph} F,$ $L$ and $M$ be nonempty closed subsets of $\SX$ and $\SY,$ respectively, such that $\cl$ and $\cm$ are convex. Then the following statements are equivalent for $c > 0$:
\begin{enumerate}[{\rm (i)}]
\item For every $a \in (0, c),$ the mapping $F$ is directionally linearly open with modulus $a$ around $(\bar{x}, \bar{y})$ with respect to $L$ and $M.$
\item There exists $\rho > 0$ such that for every $(x, y) \in \mathrm{graph}F \cap [\BX(\bar{x}, \rho) \times \BY(\bar{y}, \rho)],$ every $y^*\in\Y,$ every $x^* \in D_{L,M}^* F(x, y) (y^*),$ and every $v \in M$ there exists $u \in \cl \cap \BBX$ such that
\begin{eqnarray*}
\langle x^*, u \rangle & \leqslant & c \langle y^*, v \rangle.
\end{eqnarray*}
\end{enumerate}
\end{theorem}

\begin{proof} (i) $\Rightarrow$ (ii).
Let $a\in (0,c)$. By assumption and Corollary~\ref{Corollary32}, $F^{-1}$ is directionally Aubin continuous with modulus $\frac{1}{a}$ around $(\bar{y}, \bar{x})$ with respect to $-M$ and $L.$ 
Hence, there is $\rho>0$ such that for every $y, y'\in\BX(\bar y,2\rho)$ with $y'\in[y+\cm],\ y'\ne y $, every $ x \in F^{-1}(y) \cap \BX(\bar{x}, \rho)$, we have
\begin{eqnarray} \label{c1}
T_L(x, F^{-1}(y'))  &\leqslant& \frac{1}{a} T_{-M}(y', y) \ = \ \frac{1}{a} \|y' - y\|.
\end{eqnarray}

By definition, it suffices to prove (ii) with $D_{L,M}^* F$ being replaced by $\widehat D_{L,M}^*F.$ So let $(x, y) \in \mathrm{graph}F \cap [\BX(\bar{x}, \rho) \times \BY(\bar{y}, \rho)],$ $y^*\in{\Y},$ $x^* \in \widehat D_{L,M}^* F(x, y) (y^*),$ and $v \in M$. By definition, we have $(x^*, -y^*)  \in \widehat{N}_{L, M} (\mathrm{graph} F, ({x}, {y})).$ Take any $\epsilon \in (0,\frac{1}{a}).$ There is $\delta \in (0, \rho)$ such that for every $(x', y') \in \mathrm{graph}F \cap [(\BX({x}, \delta) \cap [x + \cl]) \times (\BY({y}, \delta) \cap [y + \cm])],$ we have
\begin{eqnarray} \label{c3}
\langle x^*, x' - {x} \rangle - \langle y^*, y' - {y} \rangle & \leqslant & \epsilon \big( \|x' - {x}\|+\|y' - {y}\| \big).
\end{eqnarray}

Let $y' \in [y + \cm]$ be such that $y'\ne y$, $\|y'-y\|<\min\{\delta, \frac{a}{2}\delta\}$ and $v=\frac{y' - {y}}{\|y' - {y}\|}.$ Then
\begin{eqnarray*}
\|y' - \bar y\| & \leqslant & \|y' - y\| + \|y - \bar y\| \ < \ \delta + \rho \ < \ 2\rho.
\end{eqnarray*}
Since the sets $\mathrm{graph} F$ and $L$ are closed, the infimum in the definition of $T_L(x, F^{-1}(y'))$ is always attained. This, together with \eqref{c1}, implies that there exists $x'\in F^{-1}(y')\cap[x+\cl]$ such that 
\begin{equation} \label{c5}
\|x' - x\|  = T_L(x, F^{-1}(y')) \leqslant \frac{1}{a} \|y' - y\|  <  \frac{1}{a}\frac{a}{2}\delta  =  \frac{\delta}{2} < \delta.
\end{equation}
Hence, by~\eqref{c3} and the first inequality in~\eqref{c5}, one has
\begin{eqnarray*}
\langle x^*, x' - {x} \rangle - \langle y^*, y' - {y} \rangle & \leqslant & \epsilon \left(\frac{1}{a} +  1 \right)\|y' - {y}\|.
\end{eqnarray*}
Equivalently, 
\begin{eqnarray*}
\left\langle x^*, \frac{x' - {x}}{\|y' - {y}\|} \right\rangle-\left\langle y^*, \frac{y' - {y}}{\|y' - {y}\|}\right\rangle &=& \left\langle x^*, \frac{x' - {x}}{\|y' - {y}\|} \right\rangle - \langle y^*, v\rangle \\ 
& \leqslant & \epsilon \left(\frac{1}{a} + 1\right).
\end{eqnarray*}
Since $\epsilon$ can be taken arbitrarily small and the sets $\mathrm{graph} F$ and $L$ are closed, in view of~\eqref{c5}, there must exists $x'\in F^{-1}(y') \cap \mathbb B^n(x, \delta)  \cap [x + \cl]$ such that 
\begin{eqnarray*}
\left\langle x^*, \frac{x' - {x}}{\|y' - {y}\|} \right \rangle \ \leqslant \ \langle y^*, v\rangle \quad \textrm{ and } \quad  \|x' - x\| \ \leqslant \ \frac{1}{a}\|y' - y\|.
\end{eqnarray*}
Furthermore, since $a$ can be taken arbitrarily close to $c$, the point $x'$ can be chosen so that $\|x' - x\| \leqslant \frac{1}{c}\|y' - y\|.$
Clearly, $u := \frac{c(x' - {x})}{\|y' - {y}\|}$ has the desired properties.

\bigskip
(ii) $\Rightarrow$ (i).
Let $a\in (0,c)$, $b\in\left(\frac{a}{a+1},\frac{c}{c+1}\right)$ and $\epsilon > 0$ be such that 
\begin{equation*}
\frac{a}{a+1} \ < \ b + \epsilon \ < \ \frac{c}{c+1} \quad \textrm{ and } \quad \frac{a\epsilon}{b} \ < \ \frac{\rho}{2}.
\end{equation*}

Take any $t\in (0,\epsilon)$ and $(x_0, y_0)\in \mathrm{graph}F \cap \left[\BX\left(\bar{x}, \frac{\rho}{2}\right) \times \BY\left(\bar{y},  \frac{\rho}{2}\right)\right].$ We will show that
\begin{eqnarray*}
\BY(y_0, at) \cap [y_0 + \cm] & \subset & F \big (\BX(x_0, t) \cap [x_0 + \cl] \big).
\end{eqnarray*}
To this end, let $y'\in \BY(y_0, at) \cap [y_0 + \cm]$ and define the function
$$f \colon \mathrm{graph} F\to\mathbb R\cup\{+\infty\}, \quad (x,y)\mapsto T_{-M}(y', y).$$
It is easy to see that $f$ is bounded from below, lower semi-continuous and finite at $(x_0, y_0) \in \mathrm{graph} F.$ 
By applying Proposition~\ref{Ekeland} for the function $f$ and the closed sets $-L$ and ${-M}$, we get $(x_b, y_b)\in \mathrm{graph} F$ such that
\begin{equation}\label{b1}
T_{-M}(y', y_b)\leqslant T_{-M}(y', y_0) - b \big(T_{-L}(x_b, x_0)+T_{-M}(y_b, y_0) \big)
\end{equation}
and for any $(x, y)\in  \mathrm{graph} F,$
\begin{equation}\label{b3}
T_{-M}(y',y_b) \leqslant T_{-M}(y',y) + b\big( T_{-L}(x,x_b)+T_{-M}(y,y_b) \big).
\end{equation}
Since $y' \in y_0 + \cm,$ we have $T_{-M}(y', y_0) = \|y' - y_0\|.$ It follows from \eqref{b1} that $T_{-M}(y',y_b)$ and $T_{-L}(x_b,  x_0) + T_{-M}(y_b, y_0)$ are finite, $(x_b, y_b)\in (x_0 + \cl, y_0 + \cm).$ 
So from~\eqref{b3}, one has
\begin{equation}\label{b5}
\|y' - y_b\| \leqslant \|y' - y_0\| - b \big ( \| x_0 - x_b\| + \| y_0 - y_b \| \big).
\end{equation}
Consequently,
\begin{eqnarray*}
\|x_0 - x_b\| + \|y_0 - y_b\| & \leqslant & \frac{1}{b}\| y_0 - y'\| \ < \ \frac{1}{b}at \ < \ \frac{a\epsilon}{b} \ < \ \frac{\rho}{2}.
\end{eqnarray*}
Hence
\begin{eqnarray*}
\|\bar x-x_b\| & \leqslant & \|\bar x - x_0\|+\| x_0 - x_b\| \ < \ \rho, \\
\|\bar y-y_b\| & \leqslant & \|\bar y- y_0\|+\| y_0 - y_b\| \ < \ \rho.
\end{eqnarray*}
Therefore $(x_b,y_b)\in \mathrm{graph} F\cap \big[(\BX(\bar x, \rho)\cap [x_0 + \cl])\times (\BY(\bar y, \rho)\cap [y_0 + \cm])\big].$

If $y_b = y',$ then from~\eqref{b5}, we have 
\begin{eqnarray*}
b\|x_b - x_0\| & \leqslant & (1-b)\|y_0 - y'\| \ < \ (1 - b)at \ < \ bt.
\end{eqnarray*}
So $x_b\in \BX(x_0, t)\cap[x_0 + \cl]$ and $y' = y_b \in F(x_b) \subset F(\BX(x_0, t)\cap[x_0 + \cl])$ which is exactly the conclusion. 

Now the statement follows if we can show that $y'=y_b$ is the only possibility.
Suppose for contradiction that $y'\ne y_b$. 
From~\eqref{b3}, it follows that the pair $(x_b,y_b)$ is a global minimizer for the function 
$$\mathrm{graph} F\to\mathbb R\cup\{+\infty\}, \ (x,y)\mapsto T_{-M}(y',y) + b \big(T_{-L}(x,x_b)+T_{-M}(y,y_b) \big).$$
Equivalently, the pair $(x_b,y_b)$ is a global minimizer for the function 
$$\X\times\Y\to\mathbb R\cup\{+\infty\}, \  (x,y)\mapsto \|y  - y'\| + b \big(\|x - x_b\| + \|y - y_b\| \big) + \iota_\Omega(x, y),$$
where we put
$$\iota_\Omega(x) := 
\begin{cases}
0 & \textrm{ if } x \in \Omega, \\
+\infty & \textrm{ otherwise,}
\end{cases}$$
which is the indicator function of the closed set
\begin{eqnarray*}
\Omega &:=&
\left\{\begin{array}{lll}
(x, y) \in \X \times \Y &:&  y\in F(x), \ y \in y' - \cm, \\
					&& 	x\in x_b+\cl, \ y \in y_b + \cm
\end{array}\right\}.
\end{eqnarray*}
Observe that the function $(x, y) \mapsto \|y - y'\| + b \big( \|x - x_b\| + \|y - y_b\| \big)$ is locally Lipschitz and the function $\iota_\Omega$ is lower semi-continuous. 
In view of Proposition~\ref{subdifferential}, these imply that
\begin{eqnarray*}
(0,0) & \in & b \BBX \times \{0\} + \{0\} \times \left(\frac{y_b - y'}{\|y_b - y'\|} + b \BBY\right) + \partial \iota_{\Omega}(x_b, y_b).
\end{eqnarray*}
(Note that $y_b \ne y'.$) On the other hand, since the cones $\cl$ and $\cm$ are closed convex, it follows easily from definitions that
\begin{eqnarray*}
\partial \iota_{\Omega}(x_b, y_b) & \subset & N_{L, M}(\mathrm{graph} F, (x_b,y_b)).
\end{eqnarray*}
Therefore, there exists $x^* \in \BBX$ and $y^* \in \BBY$ such that 
\begin{eqnarray*}
(-bx^*,v-by^*) & \in & N_{L, M}(\mathrm{graph} F, (x_b,y_b)),
\end{eqnarray*}
where $v:= -\frac{y_b - y'}{\|y_b - y'\|}\in M.$ By definition, 
\begin{eqnarray*}
-bx^* & \in &{D}^*_{L, M} F(x_b, y_b)(-v+by^*).
\end{eqnarray*}
Our assumption gives the existence of $u \in \cl \cap \BBX$ satisfying
\begin{eqnarray*}
\langle -bx^*, u \rangle & \leqslant & c \langle -v+by^*, v \rangle .
\end{eqnarray*}
Note that 
$$\langle -bx^*, u \rangle = - b\langle x^*, u \rangle \geqslant -b \ \textrm{ and } \ \langle -v + by^*, v \rangle = -1 + b \langle y^*, v\rangle \leqslant -1 + b.$$
So $-b\leqslant c(-1+b);$ i.e., $b\geqslant\frac{c}{c+1}$ which contradicts the fact $b<\frac{c}{c+1}.$
Hence we must have $y'=y_b.$
This ends the proof of the theorem.
\end{proof}

\subsection{Directional openness and variations}

Similar to the classical case (see \cite[Theorem~4.4]{Borwein1988}), a necessary and sufficient condition in terms of directional variations for a closed set-valued mapping to be directionally open is presented in the next result.

\begin{theorem}[{compare \cite[Theorem~5.2]{Frankowska2012}}] \label{Thm3}
Let $F \colon \X \rightrightarrows \Y$ be a closed set-valued mapping, 
$(\bar{x}, \bar{y}) \in \mathrm{graph} F,$ $L$ and $M$ be nonempty closed subsets of $\SX$ and $\SY,$ respectively, such that $\cl$ and $\cm$ are convex. Then the following statements are equivalent for $r \geqslant 1$:
\begin{enumerate}[{\rm (i)}]
\item The mapping $F$ is directionally open at rate $r$ around $(\bar{x}, \bar{y})$ with respect to $L$ and $M.$
\item There exists a constant $c > 0$ such that $c {\BY} \cap \cm \subset F^{(r)}_{L, M}(\bar{x}, \bar{y}).$
\end{enumerate}
\end{theorem}

In order to prove Theorem~\ref{Thm3}, we need to make use of the following lemma.

\begin{lemma}\label{Theorem11Frankowska1992}
Consider a closed set-valued mapping $F\colon \X\rightrightarrows \Y$, a closed convex cone $C \subset \X,$ and a bounded set $K \subset \Y.$
Let $(\bar{x}, \bar{y}) \in \mathrm{graph} F,\ r > 0$ and $\epsilon > 0$ be given. Assume that for some $0 \leqslant \alpha < 1,$ we have for all 
$t \in [0, \epsilon]$ and all $(x, y) \in \mathrm{graph}F \cap \big( \BBX(\bar{x}, \epsilon) \times \BBY(\bar{y}, \epsilon) \big),$ 
\begin{eqnarray}\label{a4}
y + t^r K & \subset & F \big (\BBX(x, t) \cap [x + C] \big) + \alpha t^r K.
\end{eqnarray}
Then there exists $\delta \in (0, \epsilon)$ such that for all $t \in [0, \delta]$ and all $(x, y) \in \mathrm{graph}F \cap \big( \BBX(\bar{x}, \delta) \times \BBY(\bar{y}, \delta) \big),$ 
we have
\begin{eqnarray*}
y + \big (1 - \alpha^{\frac{1}{r}} \big)^r t^r K & \subset & F \big (\BBX(x, t) \cap [x + C] \big).
\end{eqnarray*}
\end{lemma}

\begin{proof}
The proof is similar to that in \cite[Theorem~1.1]{Frankowska1992}. 
Nevertheless, for the convenience of the reader, we give a complete proof here.
Set $c:=\sup_{y\in K}\|y\|.$
Pick a constant $\delta > 0$ such that 
\begin{equation}\label{a5}
2\delta \leqslant \epsilon \quad \textrm{ and } \quad c\delta^r \big(1-\alpha^\frac{1}{r}\big)^r (1+\alpha)  + \delta \leqslant \epsilon.
\end{equation}

Let $t \in [0, \delta],$ $(x, y) \in  \mathrm{graph}F \cap \left(\BBX(\bar{x}, \delta) \times \BBY(\bar{y}, \delta)\right)$ and $y' \in y+\big (1 - \alpha^{\frac{1}{r}} \big)^r t^r K.$ 
We will construct a sequence $x_k$ converging to some point $x'\in  \BBX(x, t) \cap [x + C]$ such that $y'\in F(x').$
To this end, set ($x_0,y_0) := (x, y).$
By~\eqref{a4}, there is $(x_1, y_1)\in \mathrm{graph} F$ such that 
$$x_1\in [x_0+C],\ \ \|x_1-x_0\|\leqslant\big (1 - \alpha^{\frac{1}{r}} \big) t\ \text{ and }\ y'\in y_1+\alpha\big(\big (1 - \alpha^{\frac{1}{r}} \big) t\big)^r K.$$
Hence 
$$\|y_1- y'\|\leqslant c\alpha\big(\big (1 - \alpha^{\frac{1}{r}} \big) t\big)^r.$$
Assume that we already constructed $(x_i,y_i)\in \mathrm{graph} F,\ i = 1,\dots,k$ such that
\begin{equation}\label{a6}
x_i\in [x_{i - 1}+C] \ \text{ and }\ \|x_i - x_{i - 1}\|\leqslant\alpha^{\frac{i-1}{r}}\big(1 - \alpha^{\frac{1}{r}} \big) t
\end{equation}
and 
\begin{equation}\label{a7}
y'\in y_i+\alpha^{i}\big(\big (1 - \alpha^{\frac{1}{r}} \big) t\big)^r K.
\end{equation}
Then 
\begin{eqnarray}\label{a8}
\|x_k-x\| & \leqslant & \sum_{i = 1}^k \|x_i - x_{i-1}\| \ \leqslant \ t \big(1 - \alpha^{\frac{1}{r}} \big) \sum_{i = 1}^{k}\alpha^\frac{i - 1}{r} \ \leqslant \ t.
\end{eqnarray}
Therefore 
\begin{eqnarray*}
\|x_k - \bar x\| & \leqslant & \|x_k - x\| + \|x - \bar x\| \leqslant t + \delta \ \leqslant \ 2 \delta \ \leqslant \ \epsilon.
\end{eqnarray*}
Furthermore, by~\eqref{a5} and~\eqref{a7}, we have
$$\begin{array}{lll}
\|y_k-\bar y\|&\leqslant&\|y_k - y'\|+\|y'-y\|+\|y-\bar y\|\\
&\leqslant&c\alpha^{k}\big (1 - \alpha^{\frac{1}{r}} \big)^r t^r+c\big (1 - \alpha^{\frac{1}{r}} \big)^r t^r + \delta \\
&\leqslant&c\delta^r\big (1 - \alpha^{\frac{1}{r}} \big)^r (1+\alpha^k)+ \delta  \leqslant \epsilon.
\end{array}$$
Applying~\eqref{a4} and~\eqref{a7} to $(x_k,y_k)$, there exists $(x_{k+1},y_{k+1}) \in \mathrm{graph} F$ satisfying ~\eqref{a6} and~\eqref{a7} with $i = k + 1.$
From~\eqref{a6}, it follows that $x_k$ is a Cauchy sequence and so it converges to some $x'.$ Furthermore, \eqref{a7} implies that $\lim_{k \to \infty} y_k = y'$.
Since $F$ is a closed set-valued mapping, $y'\in F(x').$
Furthermore, by the assumption that $C$ is a closed convex cone and~\eqref{a6}, it follows that $x_k\in  [x+C]$ for all $k$, so $x'\in [x+C].$
Moreover, by taking the limit in~\eqref{a8}, we obtain $\|x'-x\|\leqslant t.$
The lemma  is proved.
\end{proof}

\begin{proof}[Proof of Theorem~\ref{Thm3}]

(i) $\Rightarrow$ (ii). 
Assume that $F$ is directionally open at rate $r$ with modulus $c > 0$ around $(\bar{x}, \bar{y})$ with respect to $L$ and $M;$ i.e., there is a constant $\rho > 0$ such that for all $t \in (0, \rho)$ and all $(x, y) \in \mathrm{graph}F\cap\left(\BX(\bar{x}, \rho) \times \BY(\bar{y}, \rho)\right),$ we have
\begin{eqnarray*}
\BY(y, ct^r) \cap [y + \cm] & \subset & F \big (\BX(x, t) \cap [x + \cl] \big).
\end{eqnarray*}
Equivalently,
\begin{eqnarray*}
c \BY\cap \cm & \subset & \frac{F  (\BX(x, t) \cap [x + \cl]) - y}{t^r},
\end{eqnarray*}
from which the desired result follows easily.

(ii) $\Rightarrow$ (i). 
Let $c > 0$ be such that $c \BY \cap \cm \subset F^{(r)}_{L, M}(\bar{x}, \bar{y}).$ By shrinking $c$ if necessary we may assume that $c \BBY \cap \cm \subset F^{(r)}_{L, M}(\bar{x}, \bar{y}).$ Pick $\alpha \in (0, 1)$ and take any $\epsilon > 0$ such that $2 \epsilon < \alpha c.$

Observe that if $v \in c \BBY \cap \cm,$ then $v \in F^{(r)}_{L, M}(\bar{x}, \bar{y})$ and so 
there exists a constant $\delta_v > 0$ such that for all $t \in (0, \delta_v)$ and all $(x, y) \in \mathrm{graph}F\cap\left(\BX(\bar{x}, \delta_v) \times \BY(\bar{y}, \delta_v)\right) ,$ we have
$$v \in \frac{F  (\BX(x, t) \cap [x + \cl]) - y}{t^r} + \epsilon (\BBY \cap \cm).$$
Hence
$$y + t^r v \in {F  (\BX(x, t) \cap [x + \cl]) } + \epsilon t^r (\BBY \cap \cm).$$
Since $\cm$ is convex, $(\BBY \cap \cm) + (\BBY \cap \cm) \subset 2(\BBY \cap \cm).$ Therefore
$$y + t^r w \in {F (\BX(x, t) \cap [x + \cl]) } + 2 \epsilon t^r (\BBY \cap \cm)$$
for all $w \in {\BBY}(v, \epsilon) \cap [v + \cm].$ 
Note that the set $c \BBY \cap \cm$ is compact, and so there exists a finite subset $\{v_1, \ldots, v_N\}$ of $c \BBY \cap \cm$ such that
$$c \BBY \cap \cm \subset \bigcup_{k = 1}^N {\BBY}(v_k, \epsilon) \cap [v_k + \cm].$$
Let $\delta := \min_{k = 1, \ldots, N} \delta_{v_k} > 0.$ Then for all $t \in (0, \delta)$ and all $(x, y) \in  \mathrm{graph}F\cap\left(\BX(\bar{x}, \delta) \times \BY(\bar{y}, \delta)\right),$ we have
$$y + t^r (c \BBY \cap \cm ) \subset {F  (\BX(x, t) \cap [x + \cl])} + 2 \epsilon t^r (\BBY \cap \cm).$$
The choice of $\epsilon$ yields 
$$\begin{array}{lll}
y + t^r (c \BBY \cap \cm ) &\subset& {F  (\BX(x, t) \cap [x + \cl])} + \alpha t^r ( c \BBY \cap \cm)\\
&\subset& {F  (\BBX(x, t) \cap [x + \cl])} + \alpha t^r ( c \BBY \cap \cm).
\end{array}$$
In view of Lemma~\ref{Theorem11Frankowska1992}, for all $t > 0$ sufficiently small and all $(x, y) \in \mathrm{graph}F$ close to $(\bar{x}, \bar{y})$
we have
$$y + (1 - \alpha^{\frac{1}{r}})^r t^r (c \BBY \cap \cm ) \subset F  (\BBX(x, t) \cap [x + \cl]).$$
Equivalently,
$$\BBY(y, c(1 - \alpha^{\frac{1}{r}})^r t^r) \cap [y + \cm] \subset F  (\BBX(x, t) \cap [x + \cl]).$$
Replacing $t$ by $(1 - \alpha^{\frac{1}{r}})t,$ we get
$$\begin{array}{lll}
\BBY\left(y, c (1 - \alpha^{\frac{1}{r}})^{2r} t^r\right) \cap [y + \cm] \!  \! \! \!  &\subset& \! \! \! \! F \left (\BBX\left(x, (1 - \alpha^{\frac{1}{r}}) t \right) \cap [x + \cl]\right)\\
&\subset& \! \! \!  \! F \left (\BX\left(x, t\right) \cap [x + \cl]\right).
\end{array}$$
Hence $F$ is directionally open at rate $r$ with modulus $c (1 - \alpha^{\frac{1}{r}})^{2r}$ around $(\bar{x}, \bar{y})$ with respect to $L$ and $M.$
\end{proof}

\begin{remark}
A useful information that may be extracted from the proof of Theorem~\ref{Thm3} is the following equality:
\begin{eqnarray*}
\bar{c} &=& \sup \left \{ c \ : \ c {\BY} \cap \cm \subset F^{(r)}_{L, M}(\bar{x}, \bar{y}) \right\},
\end{eqnarray*}
where $\bar{c} $ stands for the supremum of real numbers $c > 0$ for which $F$ is directionally open at rate $r$ with modulus $c$ around $(\bar{x}, \bar{y})$ with respect to $L$ and $M.$
\end{remark}

\subsection*{Acknowledgments}
The first author was partially supported by the Simons Foundation Targeted Grant for the  Institute of Mathematics - VAST (Award number: 558672).

\end{document}